\documentclass[12pt]{article}
\usepackage{amsfonts, amssymb, amsmath}
\usepackage{amsthm}
\usepackage[all]{xy}
\usepackage{url}
\newtheorem{theorem}{Theorem}[section]
\newtheorem{lemma}{Lemma}[section]

\newcommand{\E}{\textbf{E}}
\newcommand{\prob}{\textbf{P}}
\newcommand{\lcm}{\textrm{lcm}}

\newcommand{\N}{\mathbb{N}}

\newcommand{\dt}{\delta}

\begin{document}
\begin{center}
\textbf{\Large On the distribution of the lcm of $k$-tuples and related problems} \\
\vspace{5mm}
\begin{tabular}{c} Sungjin Kim\\[1.2\baselineskip]
\multicolumn{1}{c}{Santa Monica College, California State University Northridge} \\
 \verb+sungjin.kim@csun.edu+
\end{tabular} \end{center}

\begin{abstract}
We study the distribution of the least common multiple of positive integers in $\N\cap [1,x]$ and related problems. We refine some results of Hilberdink and T\'{o}th (2016). We also give a partial result toward a conjecture of Hilberdink, Luca, and T\'{o}th (2020).
\end{abstract}

\section{Introduction}
    Let $\N=\{1,2,3,\ldots\}$ be the set of positive integers and $\mathcal{P}=\{2,3,5,\ldots\}$ be the set of prime numbers. Let $(n_1,\ldots,n_k)=\gcd(n_1,\ldots, n_k)$ and $[n_1,\ldots, n_k]=\lcm(n_1,\ldots,n_k)$ be the greatest common divisor (gcd) and the least common multiple (lcm) of $k$-tuple of positive integers $n_1, \ldots, n_k$ respectively. A well known result is that the probability of $(n_1,n_2)=1$ for randomly chosen positive integers $n_1\leq x$ and $n_2\leq x$ tends to $1/\zeta(2)=6/\pi^2$ as $x$ tends to $\infty$. This can be made more precise by providing an asymptotic formula with an error term. Further, it is possible to write the probability distribution of the gcd of $k$-tuples of positive integers for $k\geq 2$. This is studied in~\cite[Theorem A']{FF1}. For any integer $x\geq 2$, let $(X_i^{(x)})$ be sequence of independent discrete uniform distribution (iid) on the set $\{1,\ldots, x\}$. For $1\leq m\leq x$ and $k\geq 3$,
    $$
    \prob(\gcd(X_1^{(x)}, \ldots, X_k^{(x)})=m)=\frac1{m^k\zeta(k)}+O\left(\frac1{xm^{k-1}}\right).
    $$

    We are interested in achieving a similar result for the distribution of the lcm of tuples. This also has been extensively studied in the literature. P. Diaconis and P. Erd\H{o}s~\cite[Theorem 1]{DE} provided a distribution function with an error term in the case $k=2$. For any $0\leq t\leq 1$,
    $$
    \prob(\lcm(X_1^{(x)}, X_2^{(x)})\leq tx^2)=1-\frac1{\zeta(2)}\sum_{j=1}^{\lfloor 1/t\rfloor}\frac{1-jt(1-\log(jt))}{j^2}+O_t\left(\frac{\log x}x\right).
    $$

    For $k\geq 3$, the distributions of lcm are more intricate. The case $k=3$ is handled by J. Fern\'{a}ndez, P. Fern\'{a}ndez~\cite[Theorem 3]{FF1}. For $0\leq t \leq 1$,
    $$
    \lim_{x\rightarrow\infty} \prob(\lcm(X_1^{(x)}, X_2^{(x)}, X_3^{(x)})\leq tx^3)=1-T_3\sum_{j=1}^{\infty} \frac1{j^3}\sum_{m=1}^{\infty} \frac{\Upsilon_3(m)3^{\omega(m)}}{m^2}(1-\Omega_3(tj^2m)),
    $$
    where $T_3=\prod_p \left(1-\frac1p\right)^2\left(1+\frac2p\right)$ is the probability that three random positive integers are pairwise coprime, $\Upsilon_3(m)=\prod_{p|m} \frac{1+1/p}{1+2/p}$, $\omega(m)$ is the number of distinct prime divisors of $m$, and $\Omega_k(s)$, $0<s\leq 1$ is the volume of $\{(x_1,\ldots,x_k)\in [0,1]^k : x_1\cdots x_k\leq s\}$. Note that this result does not provide an error term in the asymptotic formula.

    For $k>3$, J. Fern\'{a}ndez, P. Fern\'{a}ndez~\cite[Theorem 1]{FF1} proved that
    $$
    \prob(\lcm(X_1^{(x)},\ldots, X_k^{(x)})>tx^k)\asymp \sum_{j=1}^{\infty} \frac{1-\Omega_k(tj^{k-1})}{j^k}
    $$

    T. Hilberdink, L. T\'{o}th~\cite[Corollary 1]{HT} computed the moments of the distributions of lcm. Let $k\geq 3$ and $r>-1$. Then for every $\epsilon>0$,
    $$
    \E\left(\frac{[X_1^{(x)}, \ldots, X_k^{(x)}]^r}{x^{kr}}\right)=\frac1{x^k} \sum_{n_1,\ldots, n_k\leq x} \frac{[n_1,\ldots, n_k]^r}{x^{kr}}=\frac{C_{r,k}}{(r+1)^k}+O\left( x^{-\frac12 \mathrm{min}(r+1,1)+\epsilon}\right).
    $$
    This result combined with the method of moments~\cite[Lemma 3]{D}, prove that there is a limiting distribution of $\frac{[X_1^{(x)}, \ldots, X_k^{(x)}]}{x^k}$ as $x\rightarrow\infty$. However, an expression of the limiting distribution using the method of moments is quite complicated and it conveys no arithmetical information.

    A. Bostan, A. Marynych, K. Raschel~\cite[Theorem 2.3]{BMR} used the probabilistic method to find the limiting distribution of lcm. They proved that for $k\geq 2$,
    $$
    \frac{[X_1^{(x)}, \ldots, X_k^{(x)}]}{x^{k}}
    $$
    converges in distribution to
    $$
    \prod_{j=1}^k U_j \prod_{p\in\mathcal{P}} p^{\max\limits_{j\leq k} \mathcal{G}_j(p)-\sum\limits_{j\leq k} \mathcal{G}_j(p)}
    $$
    where $(U_j)$ is a sequence uniform distribution on $(0,1)$ and for each prime $p$, $(\mathcal{G}_j(p))$ is a sequence of random variables such that $\mathcal{G}_j(p)$ is has a geometric distribution with
    $$
    \mathcal{G}_j(p)=m \ \textrm{ with probability }\left(1-\frac1p\right)\frac1{p^m} \ \textrm{ for $m\geq 0$. }
    $$
    All distributions here are independent.

    We write
    $$
    R_k:= \prod_{p\in\mathcal{P}}p^{\max\limits_{j\leq k} \mathcal{G}_j(p)-\sum\limits_{j\leq k} \mathcal{G}_j(p)}\in 1/\N.
    $$
    Then we must have for any $0<t \leq 1$,
    \begin{equation}
    \prob\left(\frac{[X_1^{(x)}, \ldots, X_k^{(x)}]}{x^{k}}>t\right)\rightarrow \sum_{n\leq \frac1t} \int_{nt}^1 \frac{(-\log x)^{k-1}}{(k-1)!}dx \ \prob\left(R_k=\frac1n\right) \ \textrm{as $x\rightarrow\infty$.}
    \end{equation}\label{1}
    To see this, we have for each $n\in\N$, $\int_{nt}^1 (-\log x)^{k-1}/(k-1)! \ dx$ is the conditional probability that $R_k\prod_{j\leq k} U_j > t$ given that $R_k=1/n$. Note that the integral is $1-\Omega_k(nt)$ in the notation of  J. Fern\'{a}ndez and  P. Fern\'{a}ndez~\cite{FF1}. This conditional probability is $0$ if $n>1/t$. This is the reason that the sum ranges up to $n\leq 1/t$. A detailed account of the probability of $R_k=1/n$ will be described in section 3. We write $p_k(n)=\prob(R_k=1/n)$. Note that $p_2(n)=1/(n^2\zeta(2))$. We remark that this expression in (1) agrees with $r=3$ result of~\cite{FF1}. To see this, we have for any positive integer $n$,
    $$
    p_3(n)=T_3\sum_{j^2m=n}\frac{\Upsilon_3(m)3^{\omega(m)}}{j^3m^2}.
    $$

    This paper brings improvements to the current literature in the following aspects.\\

    1. Providing an error term to (\ref{1}) in accordance with~\cite[Theorem 1]{DE}.

    2. Improving the error terms of the moments of $\frac{[X_1^{(x)}, \ldots, X_k^{(x)}]}{x^{k}}$ and $\frac{[X_1^{(x)}, \ldots, X_k^{(x)}]}{X_1^{(x)}\cdots X_k^{(x)}}$ from~\cite[Corollary 1]{HT}.

    3. Improving~\cite[Theorem 4.1]{HLT} and thereby giving a partial result toward~\cite[Remark 4.2]{HLT}. \\

    All three improvements are based on the probability that $k$ positive integers $n_1,\ldots, n_k$ are pairwise relatively prime (for any $i\neq j$, $(n_i,n_j)=1$). L. T\'{o}th~\cite{T} proved that such probability tends to
    $$
    T_k=\prod_{p\in\mathcal{P}} \left(1-\frac1p\right)^{k-1}\left(1+\frac{k-1}p\right).
    $$
    An error term of $O(x^{k-1} \log^{k-1} x)$ is provided in his asymptotic formula for the counting function of such $k$ tuples. His result was successively extended by J. Hu in~\cite{H1} and~\cite{H2}. See also~\cite{RH} for an improved error term $O(x^{k-1}\log^d x)$ where $d$ is the maximum degree of the vertices of a given graph $G$. However, this improvement does not affect the error terms of our problems. We apply a modified version of~\cite[Theorem 1]{H2} which counts the number of $k$-tuples of pairwise coprime positive integers up to $x$ such that the $i$-th component is coprime to $u_i$. See section 2 for details.

    On the first aspect, we provide an error term to (\ref{1}) in accordance with~\cite[Theorem 1]{DE}. The following theorem gives a quantitative description of the distribution function of lcm as mentioned in~\cite[Page 26]{FF3}. We prove this in section 5. Our method is completely elementary.
    \begin{theorem}
    Let $k\geq 2$ be an integer. For any $0<t\leq 1$, we have
    $$\prob\left(\frac{[X_1^{(x)}, \ldots, X_k^{(x)}]}{x^{k}}>t\right)=\sum_{n\leq \frac1t} \int_{nt}^1 \frac{(-\log z)^{k-1}}{(k-1)!}dz \cdot \ p_k(n)+ O_t(x^{-1}\log^{k-1}x).$$
    \end{theorem}

    On the second aspect, we prove the following in section 6.
    \begin{theorem}
    Let $k\geq 2$ be an integer. For any $r>-1$, we have
    $$
    \E\left(\frac{[X_1^{(x)}, \ldots, X_k^{(x)}]^r}{x^{kr}}\right)=\frac1{x^k} \sum_{n_1,\ldots, n_k\leq x}\left( \frac{[n_1,\ldots, n_k]}{x^{k}}\right)^r=\frac{C_{r,k}}{(r+1)^k}+E_{r,k}(x)
    $$
    and
    $$
    \E\left(\frac{[X_1^{(x)}, \ldots, X_k^{(x)}]^r}{(X_1^{(x)}\cdots X_k^{(x)})^r}\right)=\frac1{x^k} \sum_{n_1,\ldots, n_k\leq x} \left(\frac{[n_1,\ldots, n_k]}{n_1\cdots n_k}\right)^r=C_{r,k}+E^*_{r,k}(x).
    $$
    Here, $C_{r,k}$ is the constant given in~\cite[Corollary 1]{HT}. We have an expression in a convergent Dirichlet series.
    $$
    C_{r,k}=\sum_{n=1}^{\infty} n^{-r}p_k(n), \ r>-1.
    $$
    The error terms $E_{r,k}(x)$ and $E^*_{r,k}(x)$ are both
    $$
    \begin{cases}O_r(x^{-\frac{r+1}2} \log^{\max(2^k-k-1, 2k^2-k-2)} x) \mbox{ if } -1<r\leq 1, \\
    O_r(x^{-1}\log^{k-1} x) \mbox{ if } r>1.\end{cases}
    $$
    \end{theorem}
    A significance of Theorem 1.2 lies in realizing that the constant $C_{r,k}$ gives a function of a single variable $r$ defined as a Dirichlet series. However, the series defining $C_{r,k}$ diverges at $r=-1$. Thus, the case $r=-1$ should be dealt with special care. We treat this case in section 7.

    On the third aspect, we prove the following in section 7.
    \begin{theorem}
    Let $k\geq 2$ be an integer. We have
    $$
    \log^{2^k-k-1} x\ll \E\left(\frac{X_1^{(x)}\cdots X_k^{(x)}}{[X_1^{(x)},\ldots, X_k^{(x)}]}\right)\ll \log^{2^k-k-1}x.
    $$
    \end{theorem}
    T. Hilberdink, F. Luca, L. T\'{o}th~\cite[Theorem 4.1]{HLT} proved that
    $$
    x^k\ll V_k(x):=\sum_{n_1, \ldots, n_k\leq x} \frac{n_1\cdots n_k}{[n_1,\ldots, n_k]} \ll x^k\log^{2^k -2} x.
    $$
    They also conjectured in~\cite[Remark 4.2]{HLT} that there is a positive constant $\lambda_k$ such that $V_k(x)\sim \lambda_k x^k \log^{2^k-k-1}x$ as $x\rightarrow\infty$. Our theorem shows that
    $$
    V_k(x)\asymp x^k \log^{2^k-k-1}x.
    $$
    Thus, we obtain the same order of magnitude as the conjecture.

    After the author's initial submission of this paper,  D. Essouabri, C. Salinas Zavala, L. T\'{o}th~\cite{EST} provided a proof of~\cite[Remark 4.2]{HLT} with a full asymptotic formula. They used methods from complex analysis such as multivariable Perron's formula, whereas the methods of this paper are elementary. The author provided elementary proofs of the main terms and improved the error terms by a different version of multivariable Perron's formula in~\cite{K}.
    \section{Pairwise coprime $k$-tuples}
    Let $x\geq 2$ be an integer and $\mathbf{u}=(u_1,\ldots,u_k)\in\N^k\cap [1,\infty)$. We are interested in counting the number of $k$-tuples $(n_1,\ldots,n_k)\in \N^k\cap [1,x]^k$ such that $(n_i,n_j)=1$ whenever $i\neq j$, and $(u_i,n_i)=1$ for each $i$. These are ${\rm PC}_{\mathbf{u}}$ tuples in~\cite{FF1}. J. Hu~\cite[Theorem 1]{H2} states that for any graph $G=(V,E)$ with $V=\{1,\ldots,k\}$, the number $Q_G^{\mathbf{u}}(x)$ of $k$-tuples $(n_1,\ldots, n_k)\in\N^k\cap [1,x]^k$ such that $(n_i,n_j)=1$ whenever $(i,j)\in E$, and $(u_i,n_i)=1$ for each $i$, satisfies
    $$
    Q_G^{\mathbf{u}}(x)=A_Gf_G(\mathbf{u})x^k+O(\theta(\mathbf{u})x^{k-1}\log^{k-1}x),
    $$
    where
    $$
    A_G=\prod_{p\in\mathcal{P}}\left(\sum_{m=0}^k i_m(G)\left(1-\frac1p\right)^{k-m}\frac1{p^m}\right),
    $$
    $$
    f_G(\mathbf{u})=\prod_{p|u_1u_2\cdots u_k}\left(1-\frac{\sum_{m=0}^k i_{m,S_{\mathbf{u},p}}(G)(p-1)^{k-m}}{\sum_{m=0}^k i_m(G)(p-1)^{k-m}}\right),
    $$
    and
    $$
    \theta(\mathbf{u})=\max_{i\leq k} 2^{\omega(u_i)}.
    $$
    Given a graph $G=(V,E)$ with $V=\{1,\ldots, k\}$, we say that a subset $S\subseteq V$ is {\it independent} if no two vertices of $S$ are connected by an edge in $G$. Here, $i_m(G)$ is the number of independent sets of cardinality $m$ and $i_{m,S}(G)$ is the number of independent sets of cardinality $m$ which contains a vertex in $S$. The set of indices $i$ with $1\leq i\leq k$ such that $d|u_i$ is denoted by $S_{\mathbf{u},d}$. We apply this theorem in case $G=K_k$ is the complete graph. We have $i_0(G)=1$, $i_1(G)=k$, and $i_j(G)=0$ if $j\geq 2$. Thus, we have $A_G=T_k$. If $p|u_1u_2\cdots u_k$, we have $i_{0,S_{\mathbf{u},p}}=0$, $i_{1,S_{\mathbf{u},p}}=\#\{1\leq i\leq k : p|u_i\}=|S_{\mathbf{u},p}|$, and $i_{j,S_{\mathbf{u},p}}=0$ if $j\geq 2$.
    $$
    Q_{K_k}^{\mathbf{u}}(x)=x^kT_k \prod_{p|u_1u_2\cdots u_k}\left(1-\frac{|S_{\mathbf{u},p}|}{p+k-1}\right)+O(\theta(\mathbf{u}) x^{k-1}\log^{k-1} x).
    $$
    Note that J. Fern\'{a}ndez, P. Fern\'{a}ndez~\cite[Theorem 3.1]{FF2} obtained the main term in case $\mathbf{u}$ is a pairwise coprime (PC) $k$-tuple.
    The inductive proof in~\cite{H2} readily generalizes to the counting of tuples in arbitrary cubes. Let $\mathbf{a}=(a_1,\ldots, a_k)\in \N^k$. Denote by $Q_{K_k}^{\mathbf{u}}(\mathbf{a}, x)$ the tuples in $\prod_i [1, x/a_i]$ satisfying $(n_i,n_j)=1$ whenever $(i,j)\in E$, and $(u_i,n_i)=1$ for each $i$. Then we have
    $$
    Q_{K_k}^{\mathbf{u}}(\mathbf{a}, x)=\frac{x^k}{a_1a_2\cdots a_k} T_k \prod_{p|u_1u_2\cdots u_k}\left(1-\frac{|S_{\mathbf{u},p}|}{p+k-1}\right)+O(\theta(\mathbf{u})x^{k-1}\log^{k-1}x).
    $$
    We would like to extend this asymptotic formula where each tuple is weighted by $(n_i/(x/a_i))^r$ with $r>-1$, also by the characteristic function of $(n_i/(x/a_i))_{i\leq k} \in \{(s_1,\ldots, s_k)\in [0,1]^k : s_1s_2\cdots s_k >t\}$. Let $f(\mathbf{s})=f(s_1,\ldots, s_k)$ be one of the following functions
    $$
    f(\mathbf{s})=(s_1\cdots s_k)^r, \ r>-1 \ \textrm{ or }
    $$
    $$
    f(\mathbf{s}) \ \textrm{ is the characteristic function of }\{\mathbf{s}\in [0,1]^k : s_1s_2\cdots s_k >t\}.
    $$
    \begin{lemma}
    Let $\mathbf{a}=(a_1,\ldots, a_k)\in\N^k$, $\mathbf{u}=(u_1,\ldots,u_k)\in\N^k$, and $\mathbf{s}=(s_1,\ldots, s_k)\in [0,1]^k$. Let $Q_{K_k,f}^{\mathbf{u}}(\mathbf{a},x)$ be the sum over pairwise coprime  $k$-tuples $(n_1,\ldots, n_k)$ in $\prod_{i\leq k}[1,x/a_i]$ such that $(u_i,n_i)=1$ for each $i$, and each tuple is weighted by $f((n_i/(x/a_i))_{i\leq k})$. Then for $r\geq 0$,
    \begin{equation}
    Q_{K_k,f}^{\mathbf{u}}(\mathbf{a},x)=\frac{x^kT_k}{a_1a_2\cdots a_k}\int_{[0,1]^k} f(\mathbf{s}) d\mathbf{s} \prod_{p|u_1u_2\cdots u_k}\left(1-\frac{|S_{\mathbf{u},p}|}{p+k-1}\right)+O(\theta(\mathbf{u})x^{k-1}\log^{k-1}x).
    \end{equation}\label{2}
    If $-1<r<0$, there is an additional error term
    $$
    O_r \left( x^{k-1+|r|}\left( \frac1{a_1^{|r|} a_2\cdots a_k}+ \frac1{a_1 a_2^{|r|}a_3\cdots a_k}+ \cdots + \frac1{a_1\cdots a_{k-1} a_k^{|r|}}\right)\right).
    $$
    \end{lemma}
    \begin{proof}
    For simplicity of notation, we write $x/a_i=b_i$ for each $i$. We follow the inductive proof given in~\cite{H2}. We include the proof for completeness.
    If $k=1$, the desired sum is
    $$
    \sum_{\substack{{n\leq b_1}\\{(n,u_1)=1}}} f(n/b_1).
    $$
    If $f(s)=s^r$, $r\geq 0$, then
    \begin{align*}
    \sum_{\substack{{n\leq b_1}\\{(n,u_1)=1}}} (n/b_1)^r &=\sum_{d|u_1}\mu(d)\sum_{\substack{{n\leq b_1}\\{d|n}}}\left(\frac n{b_1}\right)^r\\
    &=\sum_{d|u_1} \mu(d)\frac1{b_1^r} \sum_{k\leq \frac{b_1}d}d^rk^r=\sum_{d|u_1}\mu(d)\left(\frac d{b_1}\right)^r \left(\left(\frac{b_1}d\right)^{r+1}\frac1{r+1} + O\left(\left(\frac{b_1}d\right)^r\right)\right)\\
    &=\sum_{d|u_1}\frac{\mu(d)}d \frac{b_1}{r+1} + O\left(\sum_{d|u_1}|\mu(d)|\right)=b_1T_1 \int_0^1 s^rds \prod_{p|u_1}\left(1-\frac1p\right)  + O(2^{\omega(u_1)}).
    \end{align*}
    If $f(s)$ is the characteristic function of $(t,1]$, then
    \begin{align*}
    \sum_{\substack{{n\leq b_1}\\{(n,u_1)=1}\\{n/b_1>t}}}1&=\sum_{d|u_1}\mu(d)\sum_{\substack{{tb_1<n\leq b_1}\\{d|n}}}1\\
    &=\sum_{d|u_1}\mu(d)\sum_{\frac{tb_1}d<k\leq \frac{b_1}d}1=\sum_{d|u_1}\mu(d)\left(\frac{(1-t)b_1}d+O(1)\right)\\
    &=\sum_{d|u_1}\frac{\mu(d)}d (1-t)b_1+O\left(\sum_{d|u_1}|\mu(d)|\right)=b_1T_1\int_t^1 ds \prod_{p|u_1}\left(1-\frac1p\right)+O(2^{\omega(u_1)}).
    \end{align*}
    Thus, the results are true for both functions when $k=1$. Suppose that the result is true for some $k\geq 1$ and $f(\mathbf{s})=(s_1\cdots s_k)^r$, $r\geq 0$. That is,
    $$
    Q_{K_k,f}^{\mathbf{u}}(\mathbf{a},x)=\frac{b_1b_2\cdots b_k T_k }{(r+1)^k}  f_k(\mathbf{u})+O(\theta(\mathbf{u})x^{k-1}\log^{k-1}x),
    $$
    where
    $$
    f_k(\mathbf{u})=\prod_{p|u_1u_2\cdots u_k} \left( 1-\frac{|S_{\mathbf{u},p}|}{p+k-1}\right).
    $$
    To prove the result for $k+1$, we write
    \begin{align*}
    &Q_{K_{k+1},fs_{k+1}^r}^{\mathbf{u},u_{k+1}}(\mathbf{a},a_{k+1},x)=\sum_{\substack{{j\leq b_{k+1}}\\{(j,u_{k+1})=1}}}\left(\frac j{b_{k+1}}\right)^rQ_{K_k,f}^{j\mathbf{u}}(\mathbf{a},x)\\
    &=\sum_{\substack{{j\leq b_{k+1}}\\{(j,u_{k+1})=1}}}\left(\frac j{b_{k+1}}\right)^r\left(\frac{b_1\cdots b_k T_k}{(r+1)^k}f_k(j\mathbf{u})+O(\theta(j\mathbf{u})x^{k-1}\log^{k-1}x)\right)\\
    &=\sum_{\substack{{j\leq b_{k+1}}\\{(j,u_{k+1})=1}}}\left(\frac j{b_{k+1}}\right)^r \frac{b_1\cdots b_k T_k}{(r+1)^k}f_k(j\mathbf{u})+O(\theta(\mathbf{u})x^k\log^k x).
    \end{align*}
    Since
    $$
    f_k(j\mathbf{u})=f_k(\mathbf{u})\sum_{d|j}\mu(d)\prod_{p|d}\frac{k-|S_{\mathbf{u},p}|}{p+k-1-|S_{\mathbf{u},p}|},
    $$
    we focus on the sum over $j$. That is,
    $$
    \sum_{\substack{{j\leq b_{k+1}}\\{(j,u_{k+1})=1}}}\left(\frac j{b_{k+1}}\right)^r \sum_{d|j}\mu(d)g_k(d),
    $$
    where
    $$
    g_k(d)=\prod_{p|d}\frac{k-|S_{\mathbf{u},p}|}{p+k-1-|S_{\mathbf{u},p}|}\leq \prod_{p|d} \frac k{p+k-1}\leq \frac{k^{\omega(d)}}d
    $$
    if $d$ is square-free.

    Substituting $de=j$ in the double sum, we have
    \begin{align*}
    \sum_{\substack{{de\leq b_{k+1}}\\{(de,u_{k+1})=1}}}&\left(\frac{de}{b_{k+1}}\right)^r \mu(d)g_k(d)=\sum_{\substack{{d\leq b_{k+1}}\\{(d,u_{k+1})=1}}}\mu(d)g_k(d)\sum_{\substack{{e\leq b_{k+1}/d}\\{(e,u_{k+1})=1}}}\left(\frac d{b_{k+1}}\right)^re^r\\
    &=\sum_{\substack{{d\leq b_{k+1}}\\{(d,u_{k+1})=1}}}\mu(d)g_k(d)\left(\frac{b_{k+1}}d\frac1{r+1}f_1(u_{k+1})+O(2^{\omega(u_{k+1})})\right)\\
    &=\sum_{\substack{{d\leq b_{k+1}}\\{(d,u_{k+1})=1}}}\mu(d)g_k(d) \frac{b_{k+1}}d\frac1{r+1}f_1(u_{k+1})+O(\theta(u_{k+1})\log^kx)\\
    &=\frac{b_{k+1}f_1(u_{k+1})}{r+1}\sum_{(d,u_{k+1})=1}\frac{\mu(d)g_k(d)}d + O(\log^{k-1}x)+O(\theta(u_{k+1})\log^kx).
    \end{align*}
    The sum over $d$ is a convergent Euler product
    $$
    \prod_{p\nmid u_{k+1}}\left(1-\frac{g_k(p)}p\right).
    $$
    Rearranging the product over primes, we obtain the desired result for $k+1$ as follows.
    $$
    Q_{K_{k+1},fs_{k+1}^r}^{\mathbf{u},u_{k+1}}(\mathbf{a},a_{k+1},x)=\frac{b_1b_2\cdots b_{k+1}T_{k+1}}{(r+1)^{k+1}}f_{k+1}(\mathbf{u},u_{k+1})+O(\theta(\mathbf{u},u_{k+1}) x^k \log^k x).
    $$

    Suppose that the result is true for $k$ with $f$ is the characteristic function of the set $\{(s_1,\ldots, s_k)\in [0,1]^k \ | \ s_1 \cdots  s_k>t\}$. That is, for $0<t\leq 1$,
    $$
    Q_{K_k,f}^{\mathbf{u}}(\mathbf{a},x)=b_1b_2\cdots b_k T_k f_k(\mathbf{u})\int_{\substack{{\forall i, 0\leq s_i\leq 1}\\{s_1s_2\cdots s_k>t}}} d\mathbf{s} +O(\theta(\mathbf{u})x^{k-1}\log^{k-1}x).
    $$
    To prove the result for $k+1$, we proceed as before. Denote by $\overline{f}$ the characteristic function of the set $\{(s_1,\ldots, s_{k+1})\in [0,1]^{k+1}\ | \ s_1 \cdots  s_{k+1} >t\}$. Then we have
    \begin{align*}
    &Q_{K_{k+1},\overline{f}}^{\mathbf{u},u_{k+1}}(\mathbf{a},a_{k+1},x)\\
    &=\sum_{\substack{{j\leq b_{k+1}}\\{(j,u_{k+1})=1}}}\left(b_1\cdots b_k T_k\int_{\substack{{\forall i, 0\leq s_i\leq 1}\\{s_1\cdots s_k>t/(j/b_{k+1})}}}d\mathbf{s}f_k(j\mathbf{u})+O(\theta(j\mathbf{u})x^{k-1}\log^{k-1}x)\right)\\
    &=\sum_{\substack{{j\leq b_{k+1}}\\{(j,u_{k+1})=1}}}b_1\cdots b_k T_k\int_{\substack{{\forall i, 0\leq s_i\leq 1}\\{s_1\cdots s_k>t/(j/b_{k+1})}}}d\mathbf{s}f_k(j\mathbf{u})+O(\theta(\mathbf{u})x^k\log^k x).
    \end{align*}
    We obtain by the substitutions $de=j$ and $mv=e$,
    \begin{align*}
    \sum_{\substack{{j\leq b_{k+1}}\\{(j,u_{k+1})=1}}}&\int_{\substack{{\forall i, 0\leq s_i\leq 1}\\{s_1\cdots s_k>t/(j/b_{k+1})}}}d\mathbf{s}\sum_{d|j}\mu(d)g_k(d)=\sum_{\substack{{d\leq b_{k+1}}\\{(d,u_{k+1})=1}}}\mu(d)g_k(d)\sum_{\substack{{e\leq \frac{b_{k+1}}d}\\{(e,u_{k+1})=1}}}\int_{\substack{{\forall i, 0\leq s_i\leq 1}\\{s_1\cdots s_k>t/(de/b_{k+1})}}}d\mathbf{s}\\
    &=\sum_{\substack{{d\leq b_{k+1}}\\{(d,u_{k+1})=1}}}\mu(d)g_k(d)\sum_{m|u_{k+1}}\mu(m)\sum_{v\leq\frac{b_{k+1}}{dm}}\int_{\substack{{\forall i, 0\leq s_i\leq 1}\\{s_1\cdots s_k>t/(v/(b_{k+1}/dm))}}}d\mathbf{s}\\
    &=\sum_{\substack{{d\leq b_{k+1}}\\{(d,u_{k+1})=1}}}\mu(d)g_k(d)\sum_{m|u_{k+1}}\mu(m)\left(\frac{b_{k+1}}{dm}\int_{\substack{{\forall i, 0\leq s_i\leq 1}\\{s_1\cdots s_ks_{k+1}>t}}}d\mathbf{s}ds_{k+1}+O(1)\right)\\
    &=b_{k+1}f_1(u_{k+1})\int_{\substack{{\forall i,0\leq s_i\leq 1}\\{s_1\cdots s_ks_{k+1}>t}}}d(\mathbf{s},s_{k+1})\sum_{(d,u_{k+1})=1}\frac{\mu(d)g_k(d)}d \\&+ O(\log^{k-1}x)+O(\theta(u_{k+1})\log^kx).
    \end{align*}
    Rearranging the product over primes, we obtain the desired result for $k+1$ as follows.
    \begin{align*}
    Q_{K_{k+1},\overline{f}}^{\mathbf{u},u_{k+1}}(\mathbf{a},a_{k+1},x)=&b_1b_2\cdots b_{k+1}T_{k+1}\int_{\substack{{\forall i,0\leq s_i\leq 1}\\{s_1\cdots s_ks_{k+1}>t}}}d(\mathbf{s},s_{k+1})f_{k+1}(\mathbf{u},u_{k+1})\\&+O(\theta(\mathbf{u},u_{k+1}) x^k \log^k x).
    \end{align*}
    If $-1<r<0$, remark that
    $$
    \sum_{n\leq x} n^r = \frac1{r+1} x^{r+1} + \zeta(-r) + O(x^r).
    $$
    Applying this in the base step and the induction step of the proof, we see that the additional error term appears.
    \end{proof}
    \section{Divisibility conditions}
    Recall from~\cite{BMR} that
    for each prime $p$, $(\mathcal{G}_j(p))$ is a sequence of independent random variables such that $\mathcal{G}_j(p)$ has a geometric distribution such that
    $$
    \mathcal{G}_j(p)=m \ \textrm{ with probability }\left(1-\frac1p\right)\frac1{p^m} \ \textrm{ for $m\geq 0$,}
    $$
    and
    $$
    R_k:= \prod_{p\in\mathcal{P}}p^{\max\limits_{j\leq k} \mathcal{G}_j(p)-\sum\limits_{j\leq k} \mathcal{G}_j(p)}\in 1/\N.
    $$
    Let $X_j^{(x)}$ be the discrete uniform distribution over the set $\{1,\ldots, x\}$. The distribution $\mathcal{G}_j(p)$ is a limiting distribution (as $x\rightarrow\infty$) of
    $$\{\prob(X_j^{(x)} \textrm{ is divisible by }p^m \textrm{ and not by }p^{m+1})\}_{m\geq 0}.$$
    It follows that we also have
    $$
    \lim_{x\rightarrow\infty} \prob(X_j^{(x)} \textrm{ is divisible by }p^m)=\frac1{p^m}=\prob(\mathcal{G}_j(p)\geq m).
    $$
    Thus, we see that the value of $\mathcal{G}_j(p)$ yields a $p$-power divisibility condition on $X_j^{(x)}$. We write $p$-power divisibility condition on $k$-tuple $(X_1^{(x)},\ldots, X_k^{(x)})$ as a {\it tagged vector }$\langle e_1, \ldots, e_i\uparrow, \ldots, e_k \rangle_p$, each component may have a tag $\uparrow$. The use of uparrow notation is to indicate the particular number "or higher".  If we have the tag on a component $i$, then it means $p^{e_i}|X_i^{(x)}$: i.e. $\mathcal{G}_i(p)\geq e_i$. A component $i$ without tag means $p^{e_i}||X_i^{(x)}$: i.e. $\mathcal{G}_i(p)= e_i$. For example, a tagged vector $\langle 2\uparrow,3,4\uparrow, 6\rangle_7$ means $7^2|X_1^{(x)}$, $7^3||X_2^{(x)}$, $7^4|X_3^{(x)}$, and $7^6||X_4^{(x)}$ (or $\mathcal{G}_1(7)\geq 2$, $\mathcal{G}_2(7)= 3$, $\mathcal{G}_3(7)\geq 4$, and $\mathcal{G}_4(7)=6$). Let $S\subseteq \mathcal{P}$ be a subset of prime numbers. We may combine $p$-power divisibility conditions of $k$-tuples for $p\in S$. For such combined divisibility conditions, we write
    $$
    \wedge_{p\in S} \langle e_{p,1}, \ldots, e_{p,i}\uparrow, \ldots, e_{p,k}\rangle_p.
    $$
    For example, $\langle 2, 3\uparrow\rangle_2 \wedge \langle 4\uparrow, 5\rangle_3$ means $2^2||X_1^{(x)}$, $2^3|X_2^{(x)}$, $3^4|X_1^{(x)}$, and $3^5||X_2^{(x)}$: i.e. $\mathcal{G}_1(2)=2$, $\mathcal{G}_2(2)\geq 3$, $\mathcal{G}_1(3)\geq 4$, and $\mathcal{G}_2(3)=5$.
    For each prime $p$ with divisibility condition $\langle e_{p,1},\ldots, e_{p,k}\rangle_p$ without $\uparrow$, consider
    $$i_0=\max\{i\leq k : e_{p,i}=\max_{j\leq k} e_{p,j}\}.$$
    We call $e_{p,i_0}$ a {\it dropped maximum}. This naming is because $\max\limits_{j\leq k}\mathcal{G}_j(p)-\sum\limits_{j\leq k} \mathcal{G}_j(p)$ drops the maximum and sums up the remaining $k-1$ numbers. For
    $$i_1=\max\{i\leq k:e_{p,i}=\max_{j\leq k, j\neq i_0} e_{p,j}\},$$ we call $e_{p,i_1}$ the {\it second maximum}.

    Let $n\in\N$. We will find $\prob(R_k=\frac1n)=p_k(n)$ from the definition of $R_k$. Let $p^e||n$. Denote by $\mathbf{x}=(x_1,\ldots, x_{k-1})$ the $(k-1)$-tuple of nonnegative integers with a sum equals $e$. Let $$i_1:=i_1(\mathbf{x})=\max\{i\leq k-1:x_i=\max_{j\leq k-1} x_j\}$$ so that $x_{i_1}$ is the maximum entry of ${\bf x}$. Note that the second maximum of a valid divisibility condition is necessarily the maximum of the entries of $\mathbf{x}$. Then the following table shows all possible divisibility conditions.
    \begin{center} \begin{tabular}{| c | c | }
   \hline
   Divisibility conditions & Dropped position  \\ \hline \hline
   $\langle x_1, \ldots, x_{i_1}, x_{i_1+1},\ldots, x_{k-1}, x_{i_1}\uparrow \rangle_p$ & $k$ \\ \hline
   $\langle x_1, \ldots, x_{i_1}, x_{i_1+1},\ldots, x_{i_1}\uparrow, x_{k-1} \rangle_p$ & $k-1$  \\ \hline
   $\vdots$ & $\vdots$  \\ \hline
   $\langle x_1, \ldots, x_{i_1}, x_{i_1}\uparrow, x_{i_1+1}, \ldots, x_{k-1}\rangle_p$ & $i_1+1$  \\ \hline
   \hline
   $\langle x_1, \ldots, x_{i_1-1}, (x_{i_1}+1)\uparrow, x_{i_1}, x_{i_1+1},\ldots, x_{k-1}\rangle_p$ & $i_1$  \\ \hline
   $\langle x_1, \ldots, (x_{i_1}+1)\uparrow, x_{i_1-1},x_{i_1}, \ldots, x_{k-1}\rangle_p$ & $i_1-1$   \\ \hline
   $\vdots$ & $\vdots$ \\ \hline
   $\langle (x_{i_1}+1)\uparrow, x_1, \ldots, x_{i_1}, x_{i_1+1},\ldots, x_{k-1}\rangle_p$ & $1$  \\ \hline
 \end{tabular}
\end{center}
    The probability that $p$-part of $R_k$ equals $e$ is
    $$
    \sum_{\substack{{x_1+\cdots +x_{k-1}=e} \\{x_i\geq 0}}}\left(1-\frac1p\right)^{k-1}\frac1{p^e}\left(\frac{k-i_1(\mathbf{x})}{p^{x_{i_1}}}+\frac{i_1(\mathbf{x})}{p^{x_{i_1}+1}}\right).
    $$
    By the independence of $\mathcal{G}_j(p)$ for any prime $p$ and $j\geq 1$, we obtain
    \begin{equation}
    p_k(n)=\prod_{p\in\mathcal{P}}\sum_{\substack{{x_1+\cdots +x_{k-1}=e} \\{x_i\geq 0}}}\left(1-\frac1p\right)^{k-1}\frac1{p^e}\left(\frac{k-i_1(\mathbf{x})}{p^{x_{i_1}}}+\frac{i_1(\mathbf{x})}{p^{x_{i_1}+1}}\right).
    \end{equation}\label{3}
    Note that if a prime $q$ does not divide $n$, then $e=0$, $x_{i_1}=0$, and $i_1(\mathbf{x})=k-1$, so it contributes
    $$
    \left(1-\frac1q\right)^{k-1}\left(1+\frac{k-1}q\right).
    $$
    Therefore, the product (\ref{3}) converges. We are able to factor out $T_k$ and rewrite (\ref{3}) as
    \begin{equation}
    p_k(n)=T_k \prod_{p^e||n}\frac1{1+(k-1)/p}\sum_{\substack{{x_1+\cdots +x_{k-1}=e} \\{x_i\geq 0}}}\frac1{p^e} \left(\frac{k-i_1(\mathbf{x})}{p^{x_{i_1}}}+\frac{i_1(\mathbf{x})}{p^{x_{i_1}+1}}\right).
    \end{equation}\label{4}
    We have a single variable function $F_k(s)$ on $\Re s >-1$ defined as a Dirichlet series
    \begin{align*}
    F_k(s)&=\sum_{n=1}^{\infty} n^{-s} \ p_k(n)\\
    &=T_k \prod_{p\in\mathcal{P}}\left(1+\frac1{1+(k-1)/p}\sum_{e=1}^{\infty}\sum_{\substack{{x_1+\cdots +x_{k-1}=e} \\{x_i\geq 0}}} \frac1{p^{(s+1)e}} \left(\frac{k-i_1(\mathbf{x})}{p^{x_{i_1}}}+\frac{i_1(\mathbf{x})}{p^{x_{i_1}+1}}\right)\right).
    \end{align*}
    Note that if $r>-1$, $F_k(r)=C_{r,k}$ in~\cite[Corollary 1]{HT}. Moreover, the $p$-part of Euler product can be written in a finite sum. This is~\cite[Corollary 2.7]{BMR}.
    $$
    F_k(s)=\E R_k^s=\prod_{p\in\mathcal{P}} \frac{\left(1-\frac1p\right)^k}{\left(1-\frac 1{p^{s+1}}\right)^k}\sum_{j=1}^k\binom kj (-1)^{j-1}\frac{1-\frac1{p^{j(s+1)}}}{1-\frac 1{p^{(j-1)(s+1)+1}}}.
    $$
    It will be an interesting problem showing the two expressions of $F_k(s)$ are equivalent directly by means of rearranging sums and products. However, we did not try this here.
    \section{Distribution of lcm - main lemma}
    In~\cite{BMR}, the limiting distribution of $\frac{[X_1^{(x)},\ldots,X_k^{(x)}]}{X_1^{(x)}\cdots X_k^{(x)}}$ is proven to be $R_k$. Thus, we must have
    $$
    \lim_{x\rightarrow\infty}\prob\left(\frac{[X_1^{(x)},\ldots,X_k^{(x)}]}{X_1^{(x)}\cdots X_k^{(x)}}=\frac1n\right)=p_k(n).
    $$
    We would like to have a quantitative asymptotic formula that is uniform on both $n$ and $x$. We apply the method of section 3 to prove the following main lemma of this paper. Denote by  $P_x^{(1)}(n)$ the following item
    $$
    P_x^{(1)}(n)=\prob\left(\frac{[X_1^{(x)},\ldots,X_k^{(x)}]}{X_1^{(x)}\cdots X_k^{(x)}}=\frac1n\right).$$
    We have
    \begin{lemma}
    Let $k\geq 2$. Uniformly for $n\in\N$ and $x\geq 2$, we have
    \begin{equation}
    P_x^{(1)}(n)=p_k(n)+O\left(\tau_{k-1}(n)(2k)^{\omega(n)}x^{-1}\log^{k-1} x\right).
    \end{equation}
    \end{lemma}
    \begin{proof}
    Let $n=p_1^{e_1}\cdots p_m^{e_m}$ where $p_i$'s are distinct primes and $e_i\geq 1$. For each $p_i$, a solution of $x_1+\cdots +x_{k-1}=e_i$ gives rise to $k$ distinct divisibility conditions (see the table in section 3, each has exactly one $\uparrow$). The number of solutions of $x_1+\cdots +x_{k-1}=e_i$ is $\tau_{k-1}(p_i^{e_i})$. Thus, each prime factor of $n$ gives rise to $\tau_{k-1}(p_i^{e_i}) k$ divisibility conditions. The number of ways to combine all possible $p$-power divisibility conditions is $\prod_{p|n} (\tau_{k-1}(p_i^{e_i}) k)=\tau_{k-1}(n)k^{\omega(n)}$. Denote by $B_n$ the set of all combined divisibility conditions so that  $|B_n|=\tau_{k-1}(n)k^{\omega(n)}$. Let $M\in B_n$ be written in the following form
    \begin{center}\begin{tabular}{| c | c |}
    \hline
    Divisibility condition & Tagged ($\uparrow$) entry \\ \hline\hline
    $\langle e_{1,1}, \ldots, e_{1,k} \rangle_{p_1}$ & $e_{1,j_1} \uparrow$ \\ \hline
    $\langle e_{2,1}, \ldots, e_{2,k} \rangle_{p_2}$ & $e_{2,j_2} \uparrow$ \\ \hline
    $\vdots$ & $\vdots$ \\ \hline
    $\langle e_{m,1}, \ldots, e_{m,k} \rangle_{p_m}$ & $e_{m,j_m} \uparrow$ \\ \hline
    \end{tabular}\end{center}
    We count the number of $k$-tuples with the following conditions.
    $$
    (n_1,\ldots, n_k)\in [1,x]^k \textrm{ and $M\in B_n$ is satisfied.}  \ \ (*)
    $$
    Let
    $$
    v_1=\frac{n_1}{a_1}, \ldots,\textrm{ and } v_k=\frac{n_k}{a_k},
    $$
    where $a_1=p_1^{e_{1,1}}p_2^{e_{2,1}}\cdots p_m^{e_{m,1}}$, $a_2=p_1^{e_{1,2}}p_2^{e_{2,2}}\cdots p_m^{e_{m,2}}, \ldots,$ and $a_k=p_1^{e_{1,k}}p_2^{e_{2,k}}\cdots p_m^{e_{m,k}}$.
    Let $I_1, \ldots, I_k$ be indicator functions with
    $$
    I_j(i)=\begin{cases} 0 &\mbox{ if }j=j_i\\
    1 &\mbox{ if }j\neq j_i.\end{cases}
    $$
    Let $u_j=\prod_{i\leq m} p_i^{I_j(i)}$, $j\leq k$, $\mathbf{u}=\langle u_1,\ldots, u_k\rangle$, and $\mathbf{a}=\langle a_1,\ldots, a_k\rangle$. Then $(*)$ is equivalent to
    $$
    (v_1,\ldots, v_n)\in \prod_{i\leq k} \left[1,\frac x{a_i}\right], \forall {j\leq k}, (u_j, v_j)=1, \textrm{ and } (v_1,\ldots, v_k) \textrm{ is pairwise coprime}. \  \ (**)
    $$
    By $r=0$ case of Lemma 2.1, the number of tuples satisfying $(**)$ is
    $$
    Q_{K_k,1}^{\mathbf{u}}(\mathbf{a},x)=\frac{x^kT_k}{a_1a_2\cdots a_k}\prod_{p|u_1u_2\cdots u_k}\left(1-\frac{|S_{\mathbf{u},p}|}{p+k-1}\right)+O(\theta(\mathbf{u})x^{k-1}\log^{k-1} x).
    $$
    The set of prime factors of $u_1\cdots u_k$ and that of $n$ are identical. Since each row of $M$ has only one tagged ($\uparrow$) entry, $|S_{\mathbf{u},p_i}|=k-1$ for each $i\leq m$. Thus, the number of tuples $(n_1,\ldots, n_k)$ satisfying $(*)$ is
    $$
    \frac{x^kT_k}{a_1a_2\cdots a_k} \prod_{p|n}\left(1-\frac{k-1}{p+k-1}\right)+O(2^{\omega(n)}x^{k-1}\log^{k-1} x).
    $$
    Using $1-(k-1)/(p+k-1)=p/(p+k-1)=1/(1+(k-1)/p)$ and summing over all $M\in B_n$, we obtain the result.
    \end{proof}
    Define
    $$
    P_x^{(2)}(n)=\frac1{x^k}\sum_{\substack{{\forall {i\leq k}, n_i\leq x}\\{\frac{[n_1,\ldots,n_k]}{n_1\cdots n_k}=\frac1n}}} \left(\frac{n_1}x\right)^r\cdots \left(\frac{n_k}x\right)^r.
    $$
    We adopt notations $\mathbf{x}$, $x_{i_1}$, and $i_1(\mathbf{x})$ of section 3. Let $q_k(n)$ be the sum of the following terms arising from each $M\in B_n$,
    $$\frac1{a_1^{|r|} a_2\cdots a_k}+ \frac1{a_1 a_2^{|r|}a_3\cdots a_k}+ \cdots + \frac1{a_1\cdots a_{k-1} a_k^{|r|}}.$$
    Applying the functions $f(\mathbf{s})$ given in Lemma 2.1, we have the analogous results.
    \begin{lemma}Uniformly for $x\geq 2$ and $n\in \N$, for any $r\geq 0$, we have
    \begin{equation}P_x^{(2)}(n)=\frac1{(r+1)^k}p_k(n)+O_r\left(\tau_{k-1}(n)(2k)^{\omega(n)}x^{-1}\log^{k-1} x\right).
    \end{equation}
    If $-1<r<0$, we have (6) with an additional error term
    $$
    O_r(q_k(n) x^{-(r+1)}).
    $$
    Moreover, if $0<t\leq 1$, we have
    \begin{equation}
    \frac1{x^k}\sum_{\substack{{\forall {i\leq k}, n_i\leq x}\\{\frac{[n_1,\ldots,n_k]}{n_1\cdots n_k}=\frac1n}\\{n_1\cdots n_k>tx^k}}}1=(1-\Omega_k(t))p_k(n)+O\left(\tau_{k-1}(n)(2k)^{\omega(n)}x^{-1}\log^{k-1} x\right).
    \end{equation}
    \end{lemma}
    \begin{proof}
    In the proof of Lemma 4.1, replace $Q_{K_k,1}^{\mathbf{u}}(\mathbf{a},x)$ by $Q_{K_k,f}^{\mathbf{u}}(\mathbf{a},x)$. We have
    $$f(\mathbf{s})=(s_1\cdots s_k)^r \textrm{ for }(6) \textrm{ and }$$
    $$f(\mathbf{s}) \textrm{ is the characteristic function of the set }\{(s_1,\ldots, s_k)\in [0,1]^k:s_1\cdots s_k> t\}$$
    for (7). Then apply Lemma 2.1. The additional error term in case $-1<r<0$ is due to the error term in Lemma 2.1.
    \end{proof}
    Recall that $\Omega_k(t)$, $0<t\leq 1$ is the volume of the set
    $$\{(s_1,\ldots, s_k)\in [0,1]^k:s_1\cdots s_k\leq t\}.$$
    This volume can be written as
    $$
    \Omega_k(t)=\int_0^t \frac{(-\log z)^{k-1}}{(k-1)!}dz = \int_{-\log t}^{\infty} \frac{u^{k-1}e^{-u}}{(k-1)!}du=\sum_{j<k}\frac{t(-\log t)^j}{j!}.
    $$
    The first identity is because $(-\log z)^{k-1}/(k-1)!$ is the probability density function of $\prod_{j\leq k} U_j$ where $(U_j)$ is a sequence of the independent uniform distribution on $(0,1)$. The second identity is obtained by the change of variable $-\log z=u$. The last identity is due to the relation between Erlang distribution and Poisson distribution.

    We study the Dirichlet series $\sum_{n=1}^{\infty} n^{-s} p_k(n)$. To prove Theorem 1.2 and 1.3, we need to study the behavior of the series at $s=-1$. In the following theorem, we prove results of Tauberian type. Let $g_k(n)=T_k^{-1}p_k(n)\prod_{p|n}(1+(k-1)/p)$.
    \begin{lemma}
    The Dirichlet series $F_k(s)=\sum_{n=1}^{\infty} n^{-s} p_k(n)$ is convergent if $\Re(s)>-1$. Its Euler product is
    $$
    F_k(s)=T_k \prod_{p\in\mathcal{P}}\left(1+\frac1{1+(k-1)/p}\sum_{e=1}^{\infty}\sum_{\substack{{x_1+\cdots +x_{k-1}=e} \\{x_i\geq 0}}} \frac1{p^{(s+1)e}} \left(\frac{k-i_1(\mathbf{x})}{p^{x_{i_1}}}+\frac{i_1(\mathbf{x})}{p^{x_{i_1}+1}}\right)\right).
    $$
    Similarly, we define $G_k(s)=\sum_{n=1}^{\infty} n^{-s} g_k(n)$. This series is convergent if $\Re(s)>-1$ with an Euler product
    $$
    G_k(s)=\prod_{p\in\mathcal{P}}\left(1+\sum_{e=1}^{\infty}\sum_{\substack{{x_1+\cdots +x_{k-1}=e} \\{x_i\geq 0}}} \frac1{p^{(s+1)e}} \left(\frac{k-i_1(\mathbf{x})}{p^{x_{i_1}}}+\frac{i_1(\mathbf{x})}{p^{x_{i_1}+1}}\right)\right).
    $$
    We have $x_{i_1}\geq 1$ and
    \begin{equation}
    \sum_{e=1}^{\infty}\sum_{\substack{{x_1+\cdots +x_{k-1}=e} \\{x_i\geq 0, \ x_{i_1}=1}}} (k-i_1(\mathbf{x}))=2^k-k-1.
    \end{equation}
    Consequently, for some constants $0<c_k<d_k$,
    \begin{equation}
    \sum_{n\leq x} np_k(n)\sim c_k \log^{2^k-k-1}x \textrm{ and }\sum_{n\leq x} ng_k(n)\sim d_k \log^{2^k-k-1}x.
    \end{equation}
    \end{lemma}
    \begin{proof}
    The Euler product of $F_k(s)$ is obtained in section 3 and that of $G_k(s)$ is clear from the definition of $g_k(n)$. It is clear that $x_{i_1}\geq 1$ and $i_1\leq k-1$. If $x_{i_1}=1$, then the sum (8) is restricted to a finite sum over $e\leq k-1$. For each $i_1\leq k-1$, we have $x_{i_1+1}=\cdots=x_{k-1}=0$ and $x_j\in \{0,1\}$ for $j\leq i_1-1$. Thus, the sum (8) is
    $$
    \sum_{i_1\leq k-1} 2^{i_1-1}(k-i_1)=2^{k-2}\cdot 1+2^{k-1}\cdot 2+\cdots+2^0\cdot (k-1)=2^k-k-1.
    $$
    Consequently, we see that
    $$
    H_k(s)=F_k(s)\zeta(s+2)^{-(2^k-k-1)} \textrm{ and } I_k(s)=G_k(s)\zeta(s+2)^{-(2^k-k-1)}
    $$
    both have convergent Dirichlet series if $\Re(s)>-1-\frac1{2(k-1)}$. By Selberg-Delange method~\cite[Section 5.5, Theorem 5]{Te} or Tauberian theorem~\cite[Theorem 5.11]{MV}, we have (9) for
    $$
    c_k=\frac{H_k(-1)}{(2^k-k-1)!} \textrm{ and } d_k=\frac{I_k(-1)}{(2^k-k-1)!}.
    $$
    We have
    $$
    H_k(-1)=T_k\prod_{p\in\mathcal{P}} \left(1+\frac1{1+(k-1)/p}\sum_{e=1}^{\infty}\sum_{\substack{{x_1+\cdots +x_{k-1}=e} \\{x_i\geq 0}}}  \left(\frac{k-i_1(\mathbf{x})}{p^{x_{i_1}}}+\frac{i_1(\mathbf{x})}{p^{x_{i_1}+1}}\right)\right)\left(1-\frac1p\right)^{2^k-k-1}
    $$
    and
    $$
    I_k(-1)=\prod_{p\in\mathcal{P}}\left(1+\sum_{e=1}^{\infty}\sum_{\substack{{x_1+\cdots +x_{k-1}=e} \\{x_i\geq 0}}}   \left(\frac{k-i_1(\mathbf{x})}{p^{x_{i_1}}}+\frac{i_1(\mathbf{x})}{p^{x_{i_1}+1}}\right)\right)\left(1-\frac1p\right)^{2^k-k-1}.
    $$
    Clearly, $H_k(-1)<I_k(-1)$ and
    $$
    H_k(-1)\geq T_k\prod_{p\in\mathcal{P}}\left(1+\frac{2^k-k-1}{p+k-1}\right)\left(1-\frac1p\right)^{2^k-k-1}.
    $$
    Therefore, $0<c_k<d_k$.
    \end{proof}
    The values of $n\in\N$ in $\frac{[X_1^{(x)}, \ldots, X_k^{(x)}]}{X_1^{(x)}\cdots X_k^{(x)}}=\frac1n$ are within $1\leq n\leq x^{k-1}$. Usually the lcm is large and $n$ is small, but the large values of $n$ around $x^{k-1}$ require careful control. These large values contribute to a large error term especially when $n>x$ and $-1\leq r\leq 1$. In this regard, we would like to have uniform upper bounds in the direction of Lemma 4.1 and 4.2. The upper bounds will be useful in the proof of Theorem 1.2 and 1.3. To this end, we modify the proof of Lemma 4.1.

    \begin{lemma}
    Let $k\geq 2$. Uniformly for $n\in\N$ and $x\geq 2$, we have
    \begin{equation}
    P_x^{(1)}(n)\leq g_k(n)
    \end{equation}
    \end{lemma}
    \begin{proof}
    In the proof of Lemma 4.1, we drop the coprimality conditions and trivially bound it by the number of tuples $(v_1,\ldots, v_k)$ in $\prod_{i\leq k} \left[1,\frac x{a_i}\right]$. That is
    $$
    \prod_{i\leq k} \left\lfloor\frac x{a_i}\right\rfloor  \leq \frac{x^k}{a_1\cdots a_k}.
    $$
    Then we sum over all $M\in B_n$ to obtain the result.
    \end{proof}
    We also have similar upper bounds for weighted sums.
    \begin{lemma}
    Uniformly for $x\geq 2$ and $n\in\N$, we have
    \begin{equation}P_x^{(2)}(n)\leq g_k(n) \textrm{ if }r\geq 0.
    \end{equation}
    On the other hand, if $-1<r<0$,
    \begin{equation}P_x^{(2)}(n)\leq \frac1{(r+1)^k} \ g_k(n).
    \end{equation}
    \end{lemma}
    \begin{proof}
    In the proof of Lemma 4.1, we drop the coprimality conditions and trivially bound it by the weighted sum of tuples $(v_1,\ldots, v_k)$ in $\prod_{i\leq k} \left[1,\frac x{a_i}\right]$. If $r\geq 0$,
    $$
    \sum_{\forall {i\leq k}, v_i\leq x/a_i} \left(\frac{v_1}{x/a_1}\right)^r\cdots \left(\frac{v_k}{x/a_k}\right)^r\leq \prod_{i\leq k} \left\lfloor\frac x{a_i}\right\rfloor  \leq \frac{x^k}{a_1\cdots a_k}.
    $$
    In case $-1<r<0$, the function $s\mapsto s^r$ is decreasing, hence
    $$
    \sum_{\forall {i\leq k}, v_i\leq x/a_i} \left(\frac{v_1}{x/a_1}\right)^r\cdots \left(\frac{v_k}{x/a_k}\right)^r\leq \frac{x^k}{a_1\cdots a_k} \left(\int_0^1 s^r ds\right)^k=\frac{x^k}{a_1\cdots a_k}\frac1{(r+1)^k}.
    $$
    Then we sum over all $M\in B_n$ to obtain the result.
    \end{proof}

    \section{Distribution of lcm - proof of Theorem 1.1}
    We have the equivalence of events
    $$
    \frac{[X_1^{(x)}, \ldots, X_k^{(x)}]}{x^{k}}>t \ \Longleftrightarrow \ \exists {n\in\N}, \frac{[X_1^{(x)}, \ldots, X_k^{(x)}]}{X_1^{(x)}\cdots X_k^{(x)}}=\frac1n \textrm{ and } \frac{X_1^{(x)}\cdots X_k^{(x)}}{x^k}>nt.
    $$
    Since the event
    $$
    \frac{X_1^{(x)}\cdots X_k^{(x)}}{x^k}>nt
    $$
    is null whenever $nt>1$, we may write
    $$
    \prob\left(\frac{[X_1^{(x)}, \ldots, X_k^{(x)}]}{x^{k}}>t \right)=\sum_{n\leq \frac1t} \prob\left( \frac{[X_1^{(x)}, \ldots, X_k^{(x)}]}{X_1^{(x)}\cdots X_k^{(x)}}=\frac1n \textrm{ and } \frac{X_1^{(x)}\cdots X_k^{(x)}}{x^k}>nt\right).
    $$
    By Lemma 4.2 (7), the sum on the right-hand side is
    \begin{align*}
    \sum_{n\leq \frac1t}& \left((1-\Omega_k(nt))p_k(n)+O\left(\tau_{k-1}(n)(2k)^{\omega(n)}x^{-1}\log^{k-1} x\right)\right)\\
    &=\sum_{n\leq \frac1t}(1-\Omega_k(nt))p_k(n)+O_t(x^{-1}\log^{k-1}x)\\
    &=\sum_{n\leq \frac1t} \int_{nt}^1 \frac{(-\log z)^{k-1}}{(k-1)!}dz \cdot p_k(n)+O_t(x^{-1}\log^{k-1}x).
    \end{align*}
    This completes the proof of Theorem 1.1.
    \section{Moments of lcm - proof of Theorem 1.2}
    Let $1\leq n\leq x^{k-1}$ be an integer and $r>-1$ be a real number. we have
    $$
    \E\left(\frac{[X_1^{(x)},\ldots,X_k^{(x)}]^r}{(X_1^{(x)}\cdots X_k^{(x)})^r}\right)=\sum_{n\leq x^{k-1}} n^{-r} P_x^{(1)}(n) \textrm{ and }
    $$
    $$
    \E\left(\frac{[X_1^{(x)},\ldots,X_k^{(x)}]^r}{x^{kr}}\right)=\sum_{n\leq x^{k-1}}n^{-r}P_x^{(2)}(n).
    $$
    Suppose that $r>1$. By Lemma 4.3, 4.4, and 4.5, we have
    $$
    \sum_{x<n\leq x^{k-1}} n^{-r} P_x^{(j)}(n) \ll \sum_{x<n\leq x^{k-1}} n^{-r} g_k(n)\leq \frac1{x^{r+1}}\sum_{n\leq x^{k-1}}ng_k(n) \ll \frac{x^{\epsilon}}{x^{r+1}}\textrm{ for }j=1,2.
    $$
    Also, $\sum_{n=1}^{\infty} n^{-r}p_k(n)$ is convergent and
    $$
    \sum_{n>x} n^{-r} p_k(n) \ll x^{-(r+1)+\epsilon}.
    $$
    By Lemma 4.1, 4.2, and the convergence of $\sum_{n=1}^{\infty} n^{-r}\tau_{k-1}(n)(2k)^{\omega(n)}$, we have the result of Theorem 1.2 for $r>1$.

    Suppose that $-1<r\leq 1$. Again by Lemma 4.3, 4.4, and 4.5, we have
    $$
    \sum_{\sqrt x< n\leq x^{k-1}}n^{-r}P_x^{(j)}(n)\ll x^{-\frac{r+1}2} \sum_{n\leq x^{k-1}} n g_k(n) \ll x^{-\frac{r+1}2} \log^{2^k-k-1} x \textrm{ for }j=1,2.
    $$
    Also, $\sum_{n=1}^{\infty} n^{-r}p_k(n)$ is convergent and
    $$
    \sum_{n>\sqrt x} n^{-r}p_k(n) \ll x^{-\frac{r+1}2}\log^{2^k-k-1} x.
    $$
    The sum of the error terms of Lemma 4.1 and 4.2 over $n\leq \sqrt x$ is
    $$
    \sum_{n\leq \sqrt x}n^{-r}\tau_{k-1}(n)(2k)^{\omega(n)}x^{-1}\log^{k-1} x \ll x^{\frac{1-r}2-1}\log^{2k^2-2k-1+k-1} x = x^{-\frac{r+1}2}\log^{2k^2-k-2} x
    $$
    and the extra error term in case $-1<r<0$ contributes
    $$
    \sum_{n\leq \sqrt x} n^{-r} q_k(n) x^{-(r+1)}.
    $$
    If $-1<r<0$, the Dirichlet series $\sum n^{-s} q_k(n)$ absolutely convergent on $\Re s> r$. Thus, we have
    $$
    \sum_{n\leq \sqrt x} n^{-r} q_k(n) x^{-(r+1)}\ll \sum_{n\leq \sqrt x} n^{-r} \frac{x^{\frac{\epsilon}2}}{n^{\epsilon}} q_k(n) x^{-(r+1)}\ll x^{-(r+1)+\epsilon}
    $$
    which is $\ll x^{-\frac{r+1}2} \log^{2^k-k-1} x$.
    Therefore, we have the result of Theorem 1.2 with the error $x^{-\frac{r+1}2}\log^{\max(2^k-k-1, 2k^2-k-2)}x$.
    \section{Inverse moment of lcm - proof of Theorem 1.3}
    We begin with
    $$
    \E\left(\frac{X_1^{(x)}\cdots X_k^{(x)}}{[X_1^{(x)},\ldots,X_k^{(x)}]}\right)=\frac1{x^k}\sum_{n_1,\ldots, n_k \leq x}\frac{n_1\cdots n_k}{[n_1,\ldots, n_k]}=\sum_{n\leq x^{k-1}} n P_x^{(1)}(n).
    $$
    For the lower bound, we add the terms up to $x^{1/2-\dt}$ for $\dt>0$. Then
    $$
    \sum_{n\leq x^{1/2-\dt}}nP_x^{(1)}(n)=\sum_{n\leq x^{1/2-\dt}} np_k(n) + O\left(x^{-1}\log^{k-1}x\sum_{n\leq x^{1/2-\dt}}n\tau_{k-1}(n)(2k)^{\omega(n)}\right).
    $$
    By Lemma 4.3,
    $$
    \sum_{n\leq x^{1/2-\dt}}np_k(n)\sim c_k \log^{2^k-k-1}(x^{1/2-\dt})=c_k \left(\frac12-\dt\right)^{2^k-k-1}\log^{2^k-k-1}x.
    $$
    The error term is $O(x^{-2\dt+\epsilon})$ for any $\epsilon<2\dt$ and it is negligible. Thus, letting $\dt\rightarrow 0$, we obtain
    $$
    \liminf_{x\rightarrow\infty}\frac1{x^k\log^{2^k-k-1}x}\sum_{n_1,\ldots, n_k \leq x}\frac{n_1\cdots n_k}{[n_1,\ldots, n_k]}\geq \frac{c_k}{2^{2^k-k-1}}.
    $$
    For the upper bound, we apply Lemma 4.4. Then we have
    $$
    \sum_{n\leq x^{k-1}}nP_x^{(1)}(n)\leq \sum_{n\leq x^{k-1}} n g_k(n).
    $$
    Thus, by Lemma 4.3, we have
    $$
    \sum_{n\leq x^{k-1}}ng_k(n)\sim d_k \log^{2^k-k-1}(x^{k-1})=d_k (k-1)^{2^k-k-1} \log^{2^k-k-1}x.
    $$
    Hence, we obtain
    $$
    \limsup_{x\rightarrow\infty}\frac1{x^k\log^{2^k-k-1}x}\sum_{n_1,\ldots, n_k \leq x}\frac{n_1\cdots n_k}{[n_1,\ldots, n_k]}\leq d_k (k-1)^{2^k-k-1}.
    $$
    \textbf{Acknowledgment.}

    The author thanks the referee for pointing out a mistake in the proof of Lemma 2.1 from the initial submission of this paper. This led to the author to realize the presence of the extra error term in case $-1<r<0$.

       \flushleft

\end{document}